\documentclass[12pt,a4paper,reqno]{amsart}
\usepackage{amsopn,amssymb,mathrsfs,url,a4wide}

\parskip\medskipamount

\newtheorem{thm}{Theorem}[section]
\newtheorem{cor}[thm]{Corollary}
\newtheorem{lem}[thm]{Lemma}
\newtheorem{prop}[thm]{Proposition}

\theoremstyle{definition}
\newtheorem{defn}[thm]{Definition}

\newtheorem{rem}[thm]{Remark}

\newcommand{\ep}{\varepsilon}                                                                                       
\newcommand{\map}[3]{\ensuremath{{#1}:{#2}\to{#3}}}                                     
\newcommand{\n}[1]{\ensuremath{\|{#1}\|}}                                               

\newcommand{\p}{\partial}                                                                                               
\newcommand{\set}[2]{\ensuremath{\left\{{#1}\;:\;\,{#2}\right\}}}               
\newcommand{\su}{\subset}                                                                                               
\newcommand{\st}{($*$)}                                                                                                 
\newcommand{\dtri}{{\displaystyle |\kern-.9pt|\kern-.9pt|}}
\newcommand{\tri}{|\kern-.9pt|\kern-.9pt|}
\newcommand{\tn}[1]{\ensuremath{\tri{#1}\tri}}                                              
\newcommand{\tndot}{\ensuremath{\tri\cdot\tri}}                                             

\DeclareMathOperator{\cl}{cl}                                                                               
\DeclareMathOperator{\co}{co}                                                                               
\DeclareMathOperator{\ext}{ext}                                                                         
\DeclareMathOperator{\lsp}{span}                                                                        

\begin{document}
\title{Boundaries and polyhedral Banach spaces}

\begin{abstract}
We show that if $X$ and $Y$ are Banach spaces, where
$Y$ is separable and polyhedral, and if $\map{T}{X}{Y}$ is a
bounded linear operator such that $T^*(Y^*)$ contains a boundary $B$ of $X$, then $X$ is
separable and isomorphic to a polyhedral space. Some corollaries of this result are
presented.
\end{abstract}

\author{V.\ P.\ Fonf, R.\ J.\ Smith and S.\ Troyanski}

\address{Department of Mathematics, Ben Gurion University of the Negev, Beer-Sheva, Israel}\email{fonf@math.bgu.ac.il}
\address{School of Mathematical and Statistical Sciences, University College Dublin, Belfield, Dublin 4, Ireland} \email{richard.smith@maths.ucd.ie}
\address{Departamento de Matem\'aticas, Universidad de
Murcia, Campus de Espinardo. 30100 Espinardo (Murcia), Spain.} \email{stroya@um.es}

\thanks{The first named author is supported by Israel Science
Foundation, Grant 209/09. The second and third named authors are supported
financially by Science Foundation Ireland under Grant Number `SFI
11/RFP.1/MTH/3112'. The third named author is also supported by FEDER-MCI
MTM2011-22457 and the project of the Institute of Mathematics and Informatics,
Bulgarian Academy of Science.}
\date{\today}

\maketitle

\section{The main result}

Let $X$ be a Banach space. A subset $B\su S_{X^*}$ of the unit
sphere $S_{X^*}$ of $X$ is called a {\em boundary} of $X$ if, for any
$x\in X$, there is $f\in B$ satisfying $f(x)=\n{x}$. From the
Krein-Milman Theorem, it follows that the set $\ext B_{X^*}$ of
all extreme points of the unit ball $B_{X^*}$ of $X^*$ is a
boundary. Easy examples show that a boundary may be a proper
subset of $\ext B_{X^*}$. A separation theorem shows that, for
any boundary $B$, we have $w^*$-$\cl \co B=B_{X^*}$. It is
clear that if $X$ is infinite-dimensional, then any boundary of $X$ must be
infinite. If $X$ has a countably infinite boundary then it is separable
and isomorphically polyhedral (see \cite{F}).

The following theorem is the main result of this paper.

\begin{thm}\label{t3}
Let $X$ and $Y$ be Banach spaces, where $Y$ is separable and polyhedral.
Assume that $\map{T}{X}{Y}$ is a bounded linear operator, such
that $T^*(Y^*)$ contains a boundary $B$ of $X$. Then $X$ is
separable and isomorphic to a polyhedral space.
\end{thm}

The proof of Theorem \ref{t3} uses the following result.

\begin{prop}\label{tb}
Let $\map{T}{X}{Y}$ be a linear bounded operator, such that
$T^*(Y^*)$ contains a boundary $B$ of $X$. Then $T^*(Y^*)$ is norm
dense in $X^*$.
\end{prop}

\begin{proof} We make use of the so-called (I)-property (see \cite{FL}).
Let $K \su X^*$ be $w^*$-compact and convex, and suppose that $B\subset K$.
We say that $B$ has the (I)-property if, whenever $B=\bigcup_{i=1}^\infty B_i$,
where $B_i\subset B_{i+1}$, then the set $\bigcup_{i=1}^\infty w^*$-$\cl\co B_i$
is norm-dense in $K$. It is proved in \cite{FL} that, for $K$ as above, any
boundary $B$ of $K$ has the (I)-property ($B \su K$ is a boundary of $K$ if,
for any $x\in X$, there is $f\in B$ such that $f(x)=\max x(K)$).

Put $K=B_{X^*}$, and $B_i=T^*(iB_{Y^*})\cap B$ for $i=1,2,\dots$. Clearly,
$\bigcup_{i=1}^\infty w^*$-$\cl\co B_i \subset T^*(Y^*)$, and the result follows.
\end{proof}

\begin{proof}[Proof of Theorem \ref{t3}] Without loss of generality, we assume that $\n{T}=1$.
Since $Y$ is a separable polyhedral space, it follows that $Y^*$ is
separable (see \cite{F}), and hence by Proposition \ref{tb}, $X^*$ and $X$ are separable
too. Next, it is easily seen that $T$ is injective. Also, without loss of
generality, we can assume that $T$ is dense embedding, i.e.\ $\cl T(X)=Y$
(if necessary, pass to the subspace $Y_1 =\cl T(X)$ and the
operator $\map{T_1}{X}{Y_1}$). In particular, we can assume that $T^*$ is injective.
Now set
\[
W_n \;=\; T^{*-1}(nT^*(B_{Y^*})\cap B_{X^*}),\qquad n=1,2,\dots.
\]
Clearly,
\[
B_{Y^*}\su W_n\su nB_{Y^*},\qquad n=1,2,\dots,
\]
(recall that $\n{T}=1$).

The $W_n$ are convex centrally symmetric $w^*$-compact
bodies in $Y^*$ (with $Y$ separable and polyhedral). Choose a sequence
$\{\ep_n\}_{n=1}^\infty$ of positive numbers tending to $0$.
By \cite[Theorem 1.1]{DFH}, each such body $W_n$ can be approximated by a convex
centrally symmetric $w^*$-compact body $A_n$ having a countable
boundary, say $\{\pm h_i^n\}_{i=1}^{\infty}$, such that
\begin{enumerate}
\item[(a)] $(1-\ep_n)A_n\su W_n \su A_n$, and
\item[(b)] no $w^*$-limit point of $\{\pm
h_i^n\}_{i=1}^{\infty}$ is a support point of $A_n$, for any
$y\in Y$.
\end{enumerate}

Define
\[
U^* = w^*\text{-}\cl\co \bigcup _{n=1}^\infty (1+\ep_n)T^*(A_n),\; t_i^n =T^* h_i^n ,\; n=1,2,\dots,\; \tilde{B}=\{\pm
(1+\ep_n)t_i^n\}_{i,n=1}^{\infty}.
\]

It is easily seen that $U^*$ is a convex centrally symmetric
$w^*$-compact set in $X^*$, and $U^* =w^*$-$\cl\co \tilde{B}$.
Moreover, $B\su U^*$ and hence $B_{X^*}\su U^*$.

It follows that $U^*$, as the unit ball, defines an equivalent dual norm on
$X^*$. We denote the corresponding norm
on $X$ by $\tndot$, and show that $(X,\tndot)$ is polyhedral.
To prove this statement, it is enough to check that $\tilde{B}$
has {\st} (see \cite{FLP}), i.e.\ no $w^*$-limit point
of $\tilde{B}$ is a support point of $U^*$, for any $x\in X$.

If $f$ is a $w^*$-limit point of $\tilde{B}$ then, by the separability of $X$,
we may find a sequence of distinct points in $\tilde{B}$ that
converges to $f$ in the $w^*$-topology. There are two possibilities. The first
possibility is that, for some
fixed $n$, we can write $f=w^*$-$\lim_k (1+\ep_n)t_{i_k}^n
\in (1+\ep_n)T^* (A_n)$, where $i_1 < i_2 < \dots$. Then $T^{*-1}f$ is a $w^*$-limit point
of the set $\{\pm (1+\ep _n)h_i^n\}_{i=1}^{\infty}$ and, by (b) above, it cannot
be a support point of $(1+\ep_n)A_n$ for any $y\in Y$. Therefore,
$f$ is not a support point of $(1+\ep_n)T^* (A_n)$, for any $x\in X$.
Since $f\in (1+\ep_n)T^* (A_n)\su U^*$ it follows that $f$ is not a support
point of $U^*$, for any $x\in X$.

If the first possibility above does not hold then we can write $f=w^*$-$\lim_k (1+\ep_{n_k})t_{i_k}^{n_k}$,
where $n_k\to\infty$ as $k\to \infty$. Assume to the contrary that $f$ is a
support point of $U^*$ with $f(x_0)=\max x_0 (U^*)\neq 0$, $x_0\in X$.
By (a), it is easily seen that $f\in B_{X^*}$. Since $B_{X^*}\su
U^*$, it follows that $f(x_0)=\max x_0 (B_{X^*})=\n{x_0}\neq 0$. As
$B$ is a boundary of $X$, there is $g\in B$ satisfying $g(x_0)=\n{x_0}$.
Next, because $B\su T^* (Y^*)$, it easily follows that $B\su \bigcup_{n=1}^\infty
T^*(A_n)$, and hence $g\in T^*(A_m)$ for some $m$. However,
$(1+\ep_m)g\in U^*$ implies $\max x_0 (U^*)\geqslant (1+\ep_m)\n{x_0},$
contradicting $\n{x_0}=f(x_0)=\max x_0 (U^*)\neq 0$. The proof is complete.
\end{proof}

\section{Consequences}

\begin{cor}\label{c1}
Let $X$ be a Banach space and $Y_i$, $i=1,2,\dots$, be a sequence of
separable isomorphically polyhedral Banach spaces. Let $B$ be a boundary of $X$
and $\map{T_i}{X}{Y_i}$ a sequence of bounded linear operators, such that
\[
B \su \bigcup_{i=1}^\infty T_i^*(Y_i^*).
\]
Then $X$ is separable and isomorphically polyhedral.
\end{cor}

\begin{proof}
Without loss of generality we can assume that $Y_i$ is polyhedral and $\n{T_i}=1$ for all $i$. Let
$Y = \left(\bigoplus_{i=1}^\infty Y_i \right)_{c_0}$. Clearly $Y$ is a separable
isomorphically polyhedral space. Define $\map{T}{X}{Y}$ by $(Tx)_i = i^{-1}T_ix \in Y_i$.
Evidently, $Y^* = \left(\bigoplus_{i=1}^\infty Y_i^* \right)_{\ell_1}$ and if $f \in Y^*$ and $x \in X$, we have
\[
(T^*f)(x) \;=\; f(Tx) \;=\; \sum_{i=1}^\infty i^{-1}f_i(T_ix) \;=\; \sum_{i=1}^\infty (i^{-1}T_i^*f)(x),
\]
where $f=(f_i)_{i=1}^\infty$ and $\n{f} = \sum_{i=1}^\infty \n{f_i}$. In particular, if
$g \in Y_i^* \su Y^*$ we have $T^*g = i^{-1}T_i^* g$, whence
\[
B \su \bigcup_{i=1}^\infty T_i^*(Y_i^*) \su T^*(Y^*).
\]
Apply Theorem \ref{t3} to finish the proof.
\end{proof}

\begin{cor}\label{c2}
Assume that $X$ has a boundary $B$ which is contained in a set of the form
\[
\bigcup_{i=1}^\infty w^*\text{-}\cl\co K_i,
\]
where the $K_i$ are countable $w^*$-compact subsets of $X^*$. Then $X$ is isomorphically polyhedral.
\end{cor}

\begin{proof}
Let $Y_i$ be the isomorphically polyhedral space $C(K_i)$, and define $\map{T_i}{X}{Y_i}$
by $(T_i x)(t)=t(x)$, $x \in X$, $t \in K_i$. Then $w^*$-$\cl\co K_i \su T^*(B_{C(K_i)^*})$,
and we can apply Corollary \ref{c1}.
\end{proof}

We will see later in Remark \ref{r1} that we cannot drop the
requirement that the $K_i$ above are $w^*$-compact. Our final result
includes an application of the material above, stated in terms of M-bases.

\begin{defn}
Let $\{x_i\}\su X$ be an M-basis of a Banach space $X$, having
biorthogonal sequence $\{x_i^*\}\su X^*$. We call a subset $A\su
X^*$ {\em summable} if $\sum_{i=1}^\infty |f(x_i)|<\infty$ for all $f\in A$.
\end{defn}

\begin{lem}\label{l1}
Let $B=\{\pm f_i\}\su S_{X^*}$ be a countable boundary of $X$, take
a sequence of numbers $\{\ep_i\}$, $0<\ep_i <\frac{1}{2}$,
$\lim_i\ep_i=0$, and a sequence of vectors $\{t_i\}\su X^*$ satisfying
$\n{f_i -t_i}<\ep_i$. Then the sequence $\pm h_i =\pm (1+2\ep_i)t_i$,
$i=1,2,\dots$, is a boundary having {\st} with respect to norm it
generates, given by
\[
\tn{x}=\sup_i |h_i (x)|,\qquad x\in X.
\]
\end{lem}

\begin{proof}
First, note that for any $x\in X$, $x\neq 0$, we have

\begin{equation}\label{e1}
\tn{x}>\n{x}.
\end{equation}

Indeed, if $f_i (x)=\n{x}$ then
\begin{align*}
\tn{x} &\geqslant (1+2\ep_i)t_i (x)\\ &\geqslant (1+2\ep_i)f_i (x)-(1+2\ep_i)\n{t_i
-f_i}\,\n{x}\\
&\geqslant (1+2\ep_i)(1-\ep_i)\n{x} > \n{x}.
\end{align*}

Put $V^*=\set{f\in X^*}{\tn{f}\leqslant 1}$, $S_{V^*}=\p V^*$. By
using the Hahn-Banach Theorem and (\ref{e1}), we easily obtain
$B_{X^*}\su V^*$ and, again by using (\ref{e1}), we see that no
functional $f\in S_{V^*}\cap B_{X^*}$ (if any) attains
its norm with respect to $\tndot{}$. However, any $w^*$-limit point $g$ of $B_1$
satisfying $\tn{g}=1$ (if any) lies in $B_{X^*}$ (recall that
$\lim_i\ep_i =0$). The proof is complete.
\end{proof}

\begin{thm}\label{t4}
For a separable Banach space $X$, the following assertions are
equivalent.
\begin{enumerate}
\item[(a)] $X$ admits a boundary $B$ and a bounded linear operator
$\map{T}{X}{c_0}$, such that $T^*(c_0^*)\supset B$.
\item[(b)] $X$ admits a boundary $B$ and a bounded linear operator
$\map{T}{X}{Y}$ into a polyhedral space $Y$, such that
$T^*(Y^*)\supset B$.
\item[(c)] $X$ is isomorphically polyhedral.
\item[(d)] $X$ admits an equivalent norm having a boundary $B$, which is summable
with respect to a normalized M-basis $\{x_i\}$ with bounded biorthogonal sequence $\{x_i^*\}$.
\end{enumerate}
\end{thm}

\begin{proof}
(a) $\Rightarrow$ (b) is trivial, while
(b) $\Rightarrow$ (c) is Theorem \ref{t3}. To prove
(c) $\Rightarrow$ (d), we can assume without loss of generality that $X$ polyhedral. By
\cite{F1}, $X$ admits a countable boundary $\{f_i\}$, and $X^*$ is
separable. It is well-known (see for instance \cite[Theorem 4.59]{FHHMVZ}), that $X$
admits a (shrinking) normalized M-basis $\{x_i\}$ with bounded biorthogonal sequence $\{x_i^*\}$.
By using Lemma \ref{l1}, we easily obtain a
sequence $B=\{h_i\}\su \lsp \{x_i^*\}$ which is a boundary of
$X$ with respect to an equivalent norm. Clearly, $B$ is summable with respect to the
M-basis $\{x_i\}$. Finally, we prove (d) $\Rightarrow$ (a). Let $B$ be
a boundary of $X$ which is summable with respect to a normalized 
M-basis $\{x_i\}$ with bounded biorthogonal sequence $\{x_i^*\}$. Define $\map{T}{X}{c_0}$ by
\[
Tx \;=\; (x_i ^* (x))_{i=1}^{\infty},\qquad x\in X.
\]
Evidently, $\n{T} = \sup_i \n{x^*_i} < \infty$. If $\{e_i\}$ is the natural basis of $\ell_1 =c_0^*$, then it is easily
seen that $T^* e_i =x_i^*$, $i=1,2,\dots$. Since $B$ is summable, it
follows that $T^*(c_0^*)\supset B$.
\end{proof}

\begin{rem}\label{r1}
We cannot replace the $x^*_i$ in Theorem \ref{t4} (d) even
with a normalized basis of $X^*$, if its biorthogonal sequence does not belong to $X$.
Indeed, if $X^*$ is isomorphic to $\ell_1$ then it admits a normalized basis, with respect to which every
element of $X^*$ is an absolutely summable combination.
However, there exist non-isomorphically polyhedral Banach spaces having duals isomorphic to $\ell_1$.
The spaces in e.g.\ \cite{BD} and \cite{AH} have duals isomorphic to $\ell_1$ but do
not contain any isomorphic copies of $c_0$, and in order for a Banach space
to be isomorphically polyhedral, it is necessarily $c_0$-saturated \cite{F}. The same examples
show that it is necessary for the $K_i$ in Corollary \ref{c2} to be $w^*$-compact.
\end{rem}

\end{document}